\documentclass[11pt]{amsart}
\usepackage{hyperref}
\hypersetup{
    colorlinks=true,
    linkcolor=blue,
    citecolor=blue,
    urlcolor=blue,
    pdfauthor={Diego Corro and Adam Moreno and Masoumeh Zarei},
    pdftitle={Ricci flow from singular spaces with bounded curvature}
}
\usepackage[hyphenbreaks]{breakurl}

\usepackage[numbers]{natbib}
\usepackage{url}
\usepackage{graphicx}
\usepackage{amsfonts, amssymb, amsmath, amsthm}
\usepackage{mathtools}
\usepackage{multicol}
\usepackage{tikz}
\usepackage{tikz-cd}
\usepackage{tikz-3dplot}
\usepackage{cancel}
\usepackage{changepage}
\usepackage{csquotes}
\usepackage{adjustbox}
\usetikzlibrary{matrix}
\usepackage{enumerate}
\usepackage[normalem]{ulem}
\newcommand{\stkout}[1]{\ifmmode\text{\sout{\ensuremath{#1}}}\else\sout{#1}\fi}
\usepackage{verbatim}
\usepackage{mathscinet}

\newtheorem{mtheorem}{Theorem}

\newtheorem{theorem}{Theorem}[section]
\newtheorem{lemma}[theorem]{Lemma}
\newtheorem{prop}[theorem]{Proposition}

\newtheorem*{question}{Question}

\theoremstyle{definition}
\newtheorem{definition}[theorem]{Definition}

\newtheorem{rmk}[theorem]{Remark}

\usepackage[reftex]{theoremref}

\newcommand{\N}{\mathbb{N}} % Positive Integers
\newcommand{\R}{\mathbf{R}} %Reals
\newcommand{\C}{\mathbb{C}} %Complex

\newcommand{\Sp}{\mathbf{S}} %Sphere
\DeclareMathOperator{\curv}{curv}

\DeclareMathOperator{\diam}{diam}

\DeclareMathOperator{\Inj}{Injrad}

\DeclareMathOperator{\Sec}{Sec}
\DeclareMathOperator{\Rm}{Rm}
\DeclareMathOperator{\Ric}{Ric}

\usepackage{lineno}
\usepackage[nostamp]
{draftwatermark}
\SetWatermarkScale{5}

\usepackage%[colorinlistoftodos, textwidth=3cm]
{todonotes}

\newcounter{dccomment}

\newcounter{mzcomment}

\newcounter{amcomment}

\begin{document}
\title[Ricci flow from singular spaces with bounded curvature]{Ricci flow from singular spaces with bounded curvature}

\author[D. CORRO]{Diego Corro$^{*}$}
\address[D. CORRO]{Fakultät für Mathematik\\ 
Karlsruher Institut für Technologie\\
Karlsruhe\\
Deut\-sch\-land.}
\curraddr{School of Mathematics, Cardiff University, Cardiff, UK.}
\email{diego.corro.math@gmail.com}
\thanks{$^{*}$Supported by the DFG (281869850, RTG 2229 ``Asymptotic Invariants and Limits of Groups and Spaces''),  DGAPA postdoctoral Scholarship of the Institute of Mathematics - UNAM, UKRI Future Leaders Fellowship [grant number MR/W01176X/1; PI: J Harvey], and DFG-Eigenestelle Fellowship CO 2359/1-1}

\author[A. MORENO]{Adam Moreno}
\address[A. MORENO]{Mathematics and Statistics Department\\
Amherst College\\
220 South Pleasant Street\\
Amherst, MA 01002\\
United States of America.}
\curraddr{West Shore Community College\\
3000 N Stiles Rd\\ 
Scottville, MI 49454\\
United States of America}
\email{amoreno@amherts.edu}

\author[M. ZAREI]{Masoumeh Zarei$^\dagger$}
\address[M. ZAREI]{Fakultät Mathematik\\
 Universität	Münster\\
Münster\\
Deutschland.}
\curraddr{Fachbereich Mathematik, Universität Hamburg, Hamburg, Deutschland}
\email{masoumeh.zarei@uni-hamburg.de}
\thanks{$^\dagger$Acknowledges support  from Deutsche Forschungsgemeinschaft (DFG, German Research Foundation) under Germany's Excellence Strategy EXC 2044-390685587, Mathematics M\"unster: Dynamics-Geometry-Structure, from  DFG grant ZA976/1-1 within the Priority Program SPP2026 "Geometry at Infinity", and  partial support from  the Austrian Science Fund (FWF) [Grant DOI 10.55776/STA32]}

\subjclass[2020]{53E20, 35K40, 53C20, 53C21, 53C23}
\keywords{Ricci flow, positive sectional curvature, Alexandrov spaces, pinched spaces}

\setlength{\overfullrule}{5pt}

\begin{abstract}
We show the existence of a solution to the Ricci flow with a compact length space of bounded curvature, i.e., a space that has curvature bounded above and below in the sense of Alexandrov, as its initial condition. We show that this flow converges in the $C^{1,\alpha}$-sense to a $C^{1,\alpha}$-continuous Riemannian manifold which is isometric to the original metric space. Moreover, we prove that the flow is uniquely determined by the initial condition, up to isometry.
\end{abstract}

\maketitle

\section{Introduction}

% WHAT?

A metric flow $(M, g(t))_{(0, T)}$ on a compact manifold $M$ satisfying the Ricci-flow equation
\begin{linenomath}
\begin{equation}\label{EQ: Ricci flow}
	\frac{\partial g}{\partial t}(t) = -2\Ric\Big(g(t)\Big)\tag{\mbox{Ricci}},
\end{equation}
\end{linenomath} 
is said to have a metric space $(X, d)$ as its \emph{initial condition} if  $(X,d)$ is the Gromov-Hausdorff limit  of $(M,d_{g(t)})$ as $t\to 0$, see \cite{Richard2018}, \cite{Coffey}. The question of what assumptions on $(X, d)$ guarantee the existence of a Ricci flow $(M,g(t))$ with $(X, d)$ as its initial condition, and how further regularity assumptions on  $(X, d)$ and $(M, g(t))$ can improve the convergence  (cf. \cite[Problem~1.1]{DeruelleSchulzeSimon2022}) has been addressed in \cite{BamlerCabezas-RivasWilking2019}, \cite{Coffey}, \cite{DeruelleSchulzeSimon2022}, \cite{Richard2018}, \cite{SimonTopping2021}, \cite{Topping2021}. The study of Ricci flows with singular metric spaces as initial conditions has recently led to the resolution of Hamilton's pinching conjecture in \cite{DeruelleSchulzeSimon2022arxiv}, \cite{DeSchSi24}, \cite{LeeTopping22}.  

In this work, we consider a compact length space  $(X, d)$ whose curvature is bounded above and below in the sense of Alexandrov,  i.e., $(X,d)$ is locally (and thus globally) an Alexandrov space and locally a $\mathrm{CAT}$-space  (for more details, see \cite{BuragoBuragoIvanov}). We prove the existence and uniqueness of a solution $(M,g(t))$ to \eqref{EQ: Ricci flow} with $(X,d)$ as its initial condition. Moreover, we show that in this setting, the regularity of the convergence can be improved. Specifically, using the fact that 
$X$ is a smooth manifold and that $d$ is induced by a $C^{1,\alpha}$-continuous Riemannian metric $\hat{g}$ \cite{BerestovskijNikolaev1993}, we prove the existence of a $C^{1,\alpha}$-continuous metric $g$ on $M$ such that  $g(t)\to g$ in the $C^{1,\alpha}$-topology, and $(M,g)$ is $C^{1,\alpha}$-isometric to $(X,\hat{g})$.

\begin{mtheorem}\th\label{MT: Main Theorem}
Let $(X,d)$ be a compact length space with curvature bounded above and below in the sense of Alexandrov. Then $X$ is homeomorphic to a smooth manifold $M$,  the distance function $d$ is induced by a $C^{1,\alpha}$-continuous Riemannian metric $g$, and, moreover, there exists a unique (up to isometry) smooth Ricci flow $(M,g(t))$ 
that  converges in the $C^{1,\alpha}$-topology to $(M,g)$.  
\end{mtheorem}

%%%%%%%%%%%%%%%%%%%%%%%%%%%%%%%%%%%%%%%%%%%%%%%%%%%%%%%%%%%%

In fact, Theorem~\ref{MT: Main Theorem} follows directly from  Theorem~\ref{MT: Main Theorem 1} below. By the work of Berestovskii and Nikolaev \cite{BerestovskijNikolaev1993},   a length space $(X,d)$ with curvature bounded above and below in the sense of Alexandrov is the Gromov-Hausdorff limit of a sequence $(\bar{M}_n,\bar{g}_n)$ of manifolds with uniform upper and lower sectional curvature bounds. Moreover, in this setting, one can establish a uniform lower bound for the injectivity  radii of $(\bar{M}_n,\bar{g}_n)$. This allows us to apply Theorem~\ref{MT: Main Theorem 1}  to conclude the proof of Theorem~\ref{MT: Main Theorem}.

\begin{mtheorem}\th\label{MT: Main Theorem 1}
Let $(X,d)$ be a Gromov-Hausdorff limit of a sequence $\{(M_n,g_n)\}_{n\in \N}$ of closed Riemannian manifolds with $\delta\leq \mathrm{Sec}(g_n)\leq \Delta$ and $\mathrm{Injrad}(g_n)\geq v>0$. Then there exist a  constant $T=T(\dim X, \delta, \Delta)$, a smooth manifold $M$, and  a smooth complete solution $(M,g_\infty(t))_{t\in (0, T)}$ to \eqref{EQ: Ricci flow equation}  with $(X, d)$ as its initial condition.  Moreover, $M$ is homeomorphic to $X$, and there exists a  $C^{1,\alpha}$-continuous Riemannian metric $g$ on $M$ such that  $(M,d_{g})$ is isometric to $(X,d)$ and  $g_\infty(t)$ converges in $C^{1,\alpha}$-topology to $g$ as $t\to 0$. 
Furthermore, the Ricci flow $(M,g_\infty(t))$ depends uniquely, up to isometry, on $(X,d)$, i.e. it is independent of the choice of approximating sequence. 
\end{mtheorem}

Note that the lower bound on the sectional curvature in \th\ref{MT: Main Theorem 1} can be any real number, including negative values. 

We point out that the condition on the injectivity radius in \th\ref{MT: Main Theorem 1} is necessary to prevent collapse. For example, there is a sequence of Riemannian metrics  on the  $3$-sphere $\Sp^3$ with bounded sectional curvature converging in the Gromov-Hausdorff topology to a round $2$-sphere $\Sp^2(\frac{1}{2})\subset \R^3$ of radius $1/2$ (see \cite[Section 3.4.2]{Petersen}).

\begin{rmk}
The condition of $\mathrm{Injrad}(g_n)\geq \nu>0$ is always satisfied in even dimensions and positive pinching, i.e., $\delta>0$ by the classical work of Klingenberg \cite[Theorem 1]{Klingenberg1959}. In particular we have $\mathrm{Injrad}(g_n)\geq \pi/(2\sqrt{\Delta})>0$. The same remark applies to the other results in the introduction.
\end{rmk}

%WHY?

Before proceeding, let us add several remarks on the assumptions of Theorems~\ref{MT: Main Theorem} and \ref{MT: Main Theorem 1}. First, note that it is unclear whether the existence of an upper bound on the curvature of the length space $(X, d)$ is necessary to prove the existence of solutions to \eqref{EQ: Ricci flow} with $(X,d)$ as the initial condition. In fact, there are some studies proving the existence of solutions to \eqref{EQ: Ricci flow} whose initial data is an Alexandrov space without any upper curvature bound conditions. For instance, in dimension $2$, when $(X,d)$ is a $2$-dimensional Alexandrov space, the existence of a unique  flow $(M,g(t))$ solving \eqref{EQ: Ricci flow} with $(X,d)$ as initial condition has been proven in \cite[Corollary 0.8]{Richard2018} in the case when   $X$ is compact,  and in \cite{Coffey} for a  general $2$-dimensional Alexandrov space. In dimension $3$, similar results waere proven for polyhedra \cite{LebevaMatveevPetruninShevchishin2015} providing a resolution to the Smoothing conjecture for these spaces (see \cite[2.1 Smoothing Conjecture]{LebevaMatveevPetruninShevchishin2015}). Moreover, in \cite{GianniotisSchulze2018}, the existence of a Ricci flow with initial conditions given by a space that is a Riemannian manifold except at isolated singular points--where the geometry is that of a Euclidean cone over a smooth Riemannian manifold (depending on the point)--was proven. Even in the case when the closed Alexandrov space  $X$ is a limit of non-collapsed closed manifolds (and, by Perelman's Stability Theorem, a topological manifold), it would be of interest to  determine whether it is possible to remove the upper curvature bound requirement for obtaining a Ricci flow with $X$ as the initial condition. It is noteworthy to mention that the upper curvature condition in \th\ref{MT: Main Theorem 1} is used in our proof to apply Hamilton's compactness theorem and ensure the existence of the limit flow.

Furthermore, the assumptions about the curvature of the converging sequence in Theorem~\ref{MT: Main Theorem 1},  are of significant importance to guarantee some type of convergence of the limit Ricci flow to the initial  conditions. To elaborate on this, let $(X,d)$  be the $2$-torus with the standard flat metric. In \cite{Topping2021}, the author constructs  a sequence $\{(M_i,g_i)\}_{i\in \N}$ of metrics on the $2$-dimensional torus converging in the Gromov-Hausdorff sense to $(X,d)$. This sequence is such that the sequence  of Ricci flows $\{(M_i,g_i(t))\}_{i\in \N}$ with initial conditions $g_i(0) = g_i$ have standard curvature decay uniformly controlled by $C/t$, for some constant $C$, and they converge for all $t$ to a Riemanninan metric $g_\infty$, which is the flat metric on the torus with twice the volume of the standard flat metric. As such, the  constant Ricci flow $(M,g_\infty(t)) = (M,g_\infty)$ does not converge to $(X,d)$ in the Gromov-Hausdorff sense as $t\to 0$. Moreover, as pointed out in \cite{DeruelleSchulzeSimon2022}, given a sequence  $\{(M_i,g_i)\}_{i\in \N}$  converging  to $(X,d)$ in the Gromov-Hausdorff sense, and assuming the existence of a limit Ricci flow $(M,g_\infty(t))$ for the sequence of Ricci flows $(M_i,g_i(t))$, there may not exist a limit space for for $(M,g_\infty(t))$ as $t\to 0$. The authors then provide sufficient conditions for the existence of a limit in such  cases.  By \cite[Lemma 3.1]{SimonTopping2021} having control over the operator norm of the Riemannian curvature and a uniform lower bound for the Ricci curvature of the flow $(M,g_\infty(t))$ are sufficient conditions for the existence of a limit as $t\to 0$ (locally). Moreover, as pointed out in \cite{DeruelleSchulzeSimon2022} with this control on the Riemannian curvature tensor and the uniform lower bound of Ricci curvature, the limit of $(M,g_\infty(t))$ is isometric to $(X,d)$. 

Another result, in arbitrary dimension, concerning the existence of the Ricci flow for Gromov-Hausdorff limits of manifolds under certain curvature assumptions is given in \cite[Corollary 4]{BamlerCabezas-RivasWilking2019}. There, the authors assume some lower bounds on the curvature operator along with uniform bounds on diameter and volume. We note that \th\ref{MT: Main Theorem 1} does not follow from \cite[Corollary4]{BamlerCabezas-RivasWilking2019}, as we do not assume a uniform lower bound on the curvature operator. Moreover, in our setting, the convergence is stronger--namely, $C^{1, \alpha}$ rather than uniform convergence. 

Next, we establish an equivariant version of Theorem A, relying on results from equivariant Gromov-Hausdorff convergence (see, for example, \cite{Harvey2016}, \cite{AlAttar2024}, and \cite{Ahumada2024}).

\begin{mtheorem}\th\label{MT: Equivariant}
Let $(X,d)$ be a compact length space with curvature bounded above and below in the sense of Alexandrov, and let $G\leq \mathrm{Iso}(X,d)$ be a closed subgroup of isometries acting effectively on $X$. Then there exists a unique (up to isometry) smooth Ricci flow $(M,g(t))$ with  $(X, d)$ as its initial condition, such that   $G\leq \mathrm{Iso}(M,g(t))$ is a closed subgroup that acts effectively on $M$. Moreover,  $(M,d_{g(t)},G)$ converges in the equivariant Gromov-Hausdorff sense to $(X,d, G)$.
\end{mtheorem}

From \th\ref{MT: Equivariant} and the celebrated work of Brendle and Schoen \cite{BrendleSchoen2009} we obtain the following corollary.

\begin{mtheorem}\th\label{MC: main corollary}
Let $(X,d)$ be a compact length  space of dimension $m$, such that 
\[
\frac{1}{4}<\delta\leq \curv(X,d)\leq 1.
\]
Then $d$ is induced by a $C^{1,\alpha}$-continuous Riemannian metric, and $X$ is homeomorphic to a spherical space form. 
\end{mtheorem}

\begin{rmk}
Let $(X,d)$ be an $m$-dimensional, simply connected length space with $1/4<\delta\leq \curv(X,d)$ $\leq 1$. Then the normalized-Ricci flow given by \th\ref{MT: Main Theorem} converges in finite time to the $m$-dimensional round sphere. And given  $G\leq \mathrm{Iso}(X,d)$  any closed subgroup  acting effectively on $X$, then by \th\ref{MT: Equivariant} and \th\ref{MC: main corollary} we conclude that $G$ is a subgroup of $\mathrm{O}(m+1)$.
\end{rmk}

We point out that in the proof of \th\ref{MC: main corollary}, we can find a constant $c>0$ such that $(X,cd)$ is approximated by smooth Riemannian manifolds whose sectional curvatures lie in $(1/4,1]$ via \th\ref{T: Approximating bounded spaces}. Moreover, from \th\ref{T: Approximating bounded spaces} we also conclude that $X$ is homeomorphic to a spherical space form. 

In the case where  $1/4\leq \curv(X,d)\leq 1$, we cannot apply \th\ref{T: Approximating bounded spaces} to obtain a sequence of smooth Riemannian manifolds with sectional curvatures in $[1/4,1]$. In this case, it is still open whether an $m$-dimensional compact simply-connected length space $(X,d)$ with $1/4\leq \curv(X,d)\leq 1$  is homeomorphic to a simply-connected CROSS: $\Sp^m$, $\C P^{m/2}$, $\mathbb{H} P^{m/4}$ or the $16$-dimensional Cayley plane $CaP^2$ (see for example \cite[Theorem 2]{Tuschmann2024}). This is related to the following question:

\begin{question}[Remark 4, p. 221 in \cite{BerestovskijNikolaev1993}]
Given a compact length space $(X,d)$ of dimension $m$ such that $\delta\leq \curv(X,d)\leq \Delta$,  is it possible to find a sequence $\{(M_i,g_i)\}_{i\in \N}$ of compact $m$-dimensional Riemannian manifolds with $\delta\leq \Sec(g_i)\leq \Delta$ that converges in the Gromov-Hausdorff sense to $(X,d)$?
\end{question}

%HOW?

We now briefly outline the strategy of the proofs.  To prove different parts of  \th\ref{MT: Main Theorem 1} we proceed as follows.  First, we prove the existence of the limit Ricci flow by verifying that the conditions of Hamilton's compactness theorem (\th\ref{T: Hamiltons compactness theorem}) hold. Next, we show that the limit flow converges to a metric space isometric to the original one by satisfying the conditions of Theorem~\ref{T: Existence of limit}. To further improve the convergence, we apply Theorem~\ref{T: Compactness in C alpha topology}. The uniqueness argument follows a similar approach to that of \cite[Section 5]{BurkhardtGuim2019}.  Once \th\ref{MT: Main Theorem 1} is established, we derive \th\ref{MT: Main Theorem} from it using \th\ref{T: Approximating bounded spaces}. Theorem~\ref{MT: Equivariant}  follows from \th\ref{MT: Main Theorem 1} together with results on equivariant Gromov–Hausdorff convergence. Finally, \th\ref{MC: main corollary} is a consequence of \th\ref{MT: Main Theorem 1}, \th\ref{T: Approximating bounded spaces} and \cite[Theorem~1]{BrendleSchoen2009}.

%ORGANIZATION

The organization of our manuscript is as follows. In Section \ref{S: Preliminaries} we present different notions of convergence, results about spaces with bounded curvature, and recall some foundational theory about convergence  of sequences of Ricci flows, and estimates on the time existence for Ricci flows. In Section \ref{S: Proof of main theorem} we present the proofs of our main results.

\section*{Acknowledgements}
We thank Christoph Böhm and Felix Schulze for discussions on the proof of uniqueness.

\section{Preliminaries}\label{S: Preliminaries}

In this section, we present the necessary preliminaries required for the proof of the main theorems, for the convenience of the reader.

\subsection{Convergence of Riemannian manifolds}

In this subsection, we recall the concepts of $C^{k,\alpha}$-convergence and Cheeger-Gromov convergence and compare them to each other.

\subsubsection{$C^{k,\alpha}$-Hölder topology and convergence}

Let  $U\subset \R^m$ be a bounded open domain, $k\in \N$, and $0<\alpha\leq1$. We define the set $C^{k,\alpha}(U,\R^n)$, referred to as the  \emph{$(k,\alpha)$-Hölder space of functions}, as the set of all functions $f\in C^k(U, \R^n)$ which have finite \emph{$(k,\alpha)$-Hölder norm} defined as (see \cite[p. 301]{Petersen2006}):
\[\left\| f\right\|_{C^{k,\alpha}(U)} = 
\left\|  f\right\|_{C^k(U)}+\sum_{\vert i\vert= k}\left \|  \partial^i f\right \|_{\alpha}, 
\]
where, 
\begin{linenomath}
\begin{align*}
	\left\| f\right\|_{C^k(U)}&:=\sup_{u\in U} \left\|f(u)\right\|+\sum_{1\leq \vert i\vert\leq k}\sup_{u\in U}\left\| {\partial^i f} (u)\right\|,\\[0.7em]
\intertext{and}	
	\left \|  \partial^i f\right \|_{\alpha}&:=\sup_{u, v\in U}\frac{\| \partial^i f(u)-\partial^i f(v)\|}{\vert u-v\vert^\alpha}.
\end{align*}
\end{linenomath}
Here we use multi-index notation to define 
\begin{linenomath}
\begin{align*}
	{\partial^i f} &= \frac{\partial^l f}{\partial (x_{1})^{i_1}\cdots\partial (x_{m})^{i_m}},
\end{align*}
\end{linenomath}
with $i=(i_1, \cdots, i_n)$, and  $\vert i\vert=i_1+\cdots+ i_n=l$. Such functions $f\in C^{k,\alpha}(U,\R^n)$ are said to be \textit{$C^{k,\alpha}$-continuous}.

Consider two compact smooth manifolds $M$ and $N$ of dimensions $m$ and $n$ respectively. We define the \emph{$C^{k,\alpha}(M,N)$ space} as the set of functions $f\in C^{k}(M,N)$ such that for any chart $\phi\colon \tilde{U}\subset M\to U\subset \R^m$, any compact subset $K\subset \tilde{U}$, and any chart $\psi\colon \tilde{V}\subset N\to V\subset\R^n$ with $f(K)\subset \tilde{V}$ the map $\psi\circ f\circ\phi^{-1}\colon U\to V\subset \R^n$ belongs to the space $C^{k,\alpha}(U,\R^n)$. 

We say that a Riemannian metric $g$ on $M$ is a \emph{$C^{k,\alpha}$-continuous Riemannian metric} if for any chart $U\subset M$, the components $g_{i,j}\colon U\to \R$ are $C^{k,\alpha}$-continuous.

Let  $M$ be a $C^{k+1}$-smooth manifold and $g$ be a $C^{k,\alpha}$-continuous Riemannian metric on $M$. Further, let  $\{g^k\}_{k\in \N}$ be a sequence of $C^{k,\alpha}$-continuous Riemannian metrics on $M$. We say that \emph{the sequence $\{g^k\}$ converges in the $C^{k,\alpha}$-topology to $g$} if for any chart  $U\subset M$, the  coordinate function $g^k_{ij}$ converges to the coordinate function $g_{ij}$ in the $C^{k,\alpha}$-topology.

A sequence $\{(M_i,g_i,p_i)\}_{i\in \N}$ of pointed complete Riemannian manifolds \emph{converges in the pointed $C^{k,\alpha}$-topology} to a pointed  Riemannian manifold $(M,g,p)$ if for every $R>0$ we can find a domain $\Omega\subset M$ containing $p$ such that $B_R(p)\subset \Omega$, and embeddings $\varphi_i\colon \Omega\to M_i$ for large $i$, such that $B_R(p_i)\subset\varphi_i(\Omega)$, $\varphi_i(p) = p_i$, and $\varphi_i^\ast g_i\to g$ on $\Omega$ in the $C^{k,\alpha}$-topology (see \cite[p. 309]{Petersen2006}).

\subsubsection{Cheeger-Gromov Convergence}

Consider $\{(M_k,g_k,p_k)\}_{k\in \N}$ a sequence of smooth complete pointed Riemannian manifolds with smooth Riemannian metrics. Given $(M,g,p)$ a smooth complete pointed Riemannian manifold with a smooth Riemannian metric, we say that \emph{the sequence $\{(M_k,g_k,p_k)\}_{k\in \N}$ converges in the Cheeger-Gromov sense to $(M,g,p)$} if:
\begin{enumerate}[(i)]
\item  There exists a sequence of domains $\Omega_k\subset M$ such that for any compact subset $K\subset M$, there exists $N\in \N$ such that for $k\geq N$ it holds $K\subset \Omega_k$. In other words the sequence $\{\Omega_k\}$ \emph{exhausts} $M$.
\item For each $k\in \N$ there exists a smooth function $\phi_k\colon \Omega_k\to M_k$ such that it is a diffeomorphism onto its image, and $\phi_k(p) = p_k$.
\item The pullback metrics $\phi_k^\ast g_k$ converge to $g$ in the $C^\infty$-topology as $k\to\infty$.
\end{enumerate}

\begin{rmk}
Observe that for $M$ compact, in the definition of Cheeger-Gromov convergence,  for sufficiently large $k\in \N$ we can assume that $\Omega_k= M$.
\end{rmk}

\begin{rmk}
The Gromov-Cheeger limits are unique in the following sense: if a sequence $\{(M_i,g_i,p_i)\}$ of smooth complete Riemannian manifolds have both $(M_1,g_1,p_1)$ and $(M_2,g_2,p_2)$ as limits, then there exists a smooth isometry $I\colon (M_1,g_1)\to (M_2,g_2)$.
\end{rmk}

\begin{rmk}\th\label{R: Cheeger Gromov convergence implies C1alpha convergence}
Observe that by definition, a sequence $\{(M_k,g_k,p_k)\}$ converging to $(M,g,p)$ in the Cheeger-Gromov sense, converges in the pointed $C^{\ell,\alpha}$-topology, for any $\ell\in \N$, and $0<\alpha\leq 1$.
\end{rmk}

\subsection{Metric spaces of bounded curvature} In this section we recall the notion of curvature bounds for metric spaces. We denote by $M^2_k$ the $2$-dimensional simply-connected surface with constant sectional curvature equal to $k\in \R$. Let $(X,d_X)$ be a length space, and $x\in X$  a fixed point. Consider $y,z\in X$ close to $x$, and denote by $\triangle(x,y,z)$ the triangle with geodesic sides and vertices $x,y,z$. We say that \emph{$X$ has curvature bounded below by $k\in \R$ at $x$, (or bounded above)} if the following holds: given a geodesic triangle $\widetilde{\triangle}(\tilde{x},\tilde{y},\tilde{z})\subset M^2_k$ with sides of length equal to the ones of $\triangle(x,y,z)$, and any point $w$ on the edge $yz$ joining $y$ to $z$ in $\triangle(x,y,z)$, and $\tilde{w}$ the corresponding point in $\widetilde{\triangle}(\tilde{x},\tilde{y},\tilde{z})$ we have
\[
	d_X(x,w)\geq d_{M^2_k}(\tilde{x},\tilde{w}) \quad \mbox{(or }d_X(x,w)\leq d_{M^2_k}(\tilde{x},\tilde{w})\mbox{)}.
\]   
When $X$ has curvature bounded below (above) by $k$ at every point, we say that \emph{$X$ has curvature bounded below (above) by $k$}, and denote it by $\curv(X) \geq k$ (or $\curv(X)\leq k$).

The following theorem tells us that length spaces with both an upper and lower curvature bound are smooth manifolds, and their geometry is determined by a $C^1$ regular Riemannian metric.

\begin{theorem}\cite[Theorem 15.1 and Remarks right after] {BerestovskijNikolaev1993}\th\label{T: Approximating bounded spaces}
Consider a compact length space $(X,d_X)$ with
\[
k\leq \curv(X,d_X)\leq K,
\]
and  any $\varepsilon>0$. Then $X$ is a smooth manifold, and there exists an $0<\alpha<1$ and a sequence of  smooth Riemannian metrics $\{g^\varepsilon_i\}_{i\in \N}$ on $X$ with
\[
	k-\varepsilon \leq \mathrm{Sec}(g^\varepsilon_i) \leq K+\varepsilon,
\]
which $C^{1,\alpha}$-converges to a $C^{1,\alpha}$-Riemannian metric $g^\varepsilon$ on $X$, with $d_{g^\varepsilon}=d_X$. Moreover the distance functions $d_{g_i}$ converge uniformly to the distance $d_X$.
\end{theorem}

\begin{rmk}\th\label{R: Approximation for spaces of bounded curvature is also Gromov-Hausdorff}
When we consider a sequence $\{(M,g_i)\}_{i\in \N}$ of manifolds given by \th\ref{T: Approximating bounded spaces}, as  metric spaces $(M,d_{g_i})$, these converge in the Gromov-Hausdorff sense to $(X,d_g)$ since the $C^{1,\alpha}$-convergence of the Riemannian metric implies a $C^1$-convergence, and this in turn implies the convergence in the Gromov-Hausdorff sense. 
\end{rmk}

In what follows, we state several lemmata that are likely known to the experts, although we could not find explicit proofs for them. Therefore, for the sake of completeness we provide proofs for each one.

First, we show that the upper bound on the sectional curvature of a compact manifold imposes restrictions on the injectivity radius of the manifold.  Recall that a \emph{geodesic loop}  $\gamma\colon [0,1]\to M$  is a smooth minimizing geodesic with  $\gamma(0)=\gamma(1)$ and $\gamma'(0)=-\gamma'(1)$.

\begin{lemma}\th\label{L: injectivity radius estimate}
Let $(M,g)$ be a compact Riemannian manifold such that there exists $\Delta>0$ constant with $\Sec(M)\leq \Delta$. Then we have that either $\Inj(M)\geq \pi/\sqrt{\Delta}$, or 
\[
\Inj(M) = \frac{1}{2}\big(\mbox{length of shortest geodesic loop in }M\big).
\]
\end{lemma}

\begin{proof}
Let $p\in M$ be such that $\Inj(p) = \Inj(M)$. We observe that by definition $\Inj(p)= d(p,C(p))$, where $C(p)$ is the cut locus of $p$.  Let $q\in C(p)$ be such that $d(p,q) = d(p,C(p))$. 

We assume now that $\Inj(M)<\pi/\sqrt{\Delta}$, and consider $c\colon[0,1]\to M$ any minimizing geodesic from $p$ to $q$. Given our assumption, by \cite[Theorem 2.6.2 (i)]{Klingenberg} we have that there are no conjugate points to $p$ on $c$. Then by \cite[Lemma 2.1.11 (iii)]{Klingenberg} there exists exactly one minimizing geodesic $c_1\colon [0,1]\to M$ from $p$ to $q$ not equal to $c$, such that $c'(1)=-(c_1)'(1)$ and $c'(0) =-(c_1)'(0)$. We consider the geodesic loop $\gamma$ given by concatenating $c$ and $c_1$. We observe that $\gamma$ has length $2\Inj(M)$ as desired.

Assume now that there exists another geodesic loop $\tilde{\gamma}$ with a shorter length, i.e. $l(\tilde{\gamma})<l(\gamma)$. Consider $p'$ and $q'$ in $\tilde{\gamma}$ such that $d(p',q')= l(\tilde{\gamma})/2$. Then we observe that $q'\in C(p')$. Thus, $$\Inj(p')=d(C(p'), p')\leq d(p', q')=l(\tilde{\gamma})/2<l(\gamma)/2 = \Inj(M),$$ which yields a contradiction, hence the result. 

\end{proof}

The next lemma provides a uniform lower bound on the volume of unit balls in a non-collapsing sequence of Riemannian manifolds with a uniform lower bound on sectional curvature.

\begin{lemma}\th\label{L: Non-collapsing dimension and lower curvature bound implies lower bound on volume of unit balls}
Consider  a sequence  $(M_i,g_i)$ of compact  smooth $n$-dimensional Riemannian manifolds with $\Sec\geq \delta$. Suppose that there exists a compact Alexandrov space  $(X,d)$ of dimension $n$ such that $(M_i,d_{g_i})$ converges in the Gromov-Hausdorff sense to $(X,d)$. Then there exists a constant $\nu_0>0$ such that, for sufficiently large $i$, the volume of the unit balls $B_{1}(p_i)\subset(M_i,g_i)$ are uniformly bounded below by $\nu_0$. That is there exists $N\in \N$ such that for all $p_i\in M_i$ and $i\geq N$ we have
\[
\mathrm{vol}(g_i)(B_{1}(p_i))\geq \nu_0>0.
\]
\end{lemma}
\begin{proof}
We recall (see \cite[Proposition 12.6]{Taylor}) that the Hausdorff $n$-measure $\mathcal{H}^n(M_i,d_{g_i})$ of\linebreak$(M_i,d_{g_i})$ coincides with the volume measure of $(M_i,g_i)$, that is
\[
\mathcal{H}^n(M_i,d_{g_i}) = \mathrm{vol}(M_i,g_i):=\int_M \,d\mathrm{vol}(g_i).
\]
Moreover, since every space $(M_i,d_{g_i})$ and $(X,d)$ are compact $n$-dimensional Alexandrov spaces with $\curv\geq \delta$, and $(M_i,d_{g_i})$ converges in the Gromov-Hausdorff sense to $(X,d)$, then by \cite[Theorem 10.8]{BuragoGromovPerelman1992} we have that
\[
\mathcal{H}^n(M_i,d_{g_i})\underset{i\to\infty}{\longrightarrow} \mathcal{H}^n(X,d)>0.
\]
Thus, we conclude that there exist $\nu>0$ such that for all  $i\in \N$ we have $\mathrm{vol}(M_i,g_i)\geq \nu$.

Let us denote by $\mathcal{M}^n_\delta$ the complete simply-connected $n$-dimensional Riemannian manifold with constant sectional curvature equal to $\delta$. Then since $\Sec(g_i)\geq \delta$, by the Bishop-Gromov Comparison Theorem (see \cite[Section 9.1.3 Lemma 1.6]{Petersen}) we have that
\[
\frac{\mathrm{vol}(g_i)(B_1(p_i))}{\mathrm{vol}(\mathcal{M}^n_\delta)(B_1(\bar{p}_i))}\geq \frac{\mathrm{vol}(g_i)(B_{D_i}(p_i))}{\mathrm{vol}(\mathcal{M}^n_\delta)(B_{D_i}(\bar{p}_i))} = \frac{\mathrm{vol}(M,g_i)}{\mathrm{vol}(\mathcal{M}^n_\delta)(B_{D_i}(\bar{p}_i))}\geq \frac{\nu}{\mathrm{vol}(\mathcal{M}^n_\delta)(B_{D_i}(\bar{p}_i))},
\]
where $D_i = \mathrm{diam}(M_i,g_i)$, and $\mathrm{vol}(\mathcal{M}^n_\delta)(B_{R}(\bar{p}_i))$ denotes the volume in $\mathcal{M}^n_\delta$ of the ball of radius $R$, $B_{R}(\bar{p}_i)\subset \mathcal{M}^n_\delta$ centered at a comparison point $\bar{p}_i\in \mathcal{M}^n_\delta$.

Since $X$ is compact and $(M_i,d_{g_i})$ converges to $(X,d)$ in the Gromov-Hausdorff sense, then by \cite[Exercise 7.3.14]{BuragoBuragoIvanov} we have that 
\[
D_i\underset{i\to\infty}{\longrightarrow} D= \mathrm{diam}(X,d).
\]
Thus for  $\epsilon'>0$ fixed, there exists $N\in \N$ such that $D_i\in (D-\epsilon',D+\epsilon')$. This implies that 
\[
\mathrm{vol}(\mathcal{M}^n_\delta)(B_{D_i}(\bar{p}_i))\leq \mathrm{vol}(\mathcal{M}^n_\delta)(B_{D+\epsilon'}(\bar{p}_i)).
\]
Thus we conclude that for $i\geq N$ we have
\[
\mathrm{vol}(g_i)(B_1(p_i))\geq \frac{\nu \mathrm{vol}(\mathcal{M}^n_\delta)(B_1(\bar{p}_i))}{\mathrm{vol}(\mathcal{M}^n_\delta)(B_{D+\epsilon'}(\bar{p}_i))} := \nu_0>0
\]
We observe that $\nu_0$ does not depend on choice of a comparison point $\bar{p}_i\in \mathcal{M}^n_\delta$. This is because the space $\mathcal{M}^n_\delta$  is homogeneous, and thus the volumes of $B_1(\bar{p}_i),\: B_{D+\epsilon'}(\bar{p}_i)\subset \mathcal{M}^n_\delta$ do not depend on the choice of the base point. Thus $\nu_0$ is a constant.
\end{proof}

\begin{lemma}\th\label{L: Bounded curvature and limit n-dimensional implies bound injectivity radius}
Consider a sequence $(M_i,g_i)$  of compact smooth $n$-dimensional Riemannian manifolds with $\delta\leq  \Sec\leq \Delta$. Suppose that there exists a compact Alexandrov space $(X,d)$  of dimension $n$ such that $(M,d_{g_i})$ converges  to $(X,d)$ in the Gromov-Hausdorff sense. Then there exists a constant $i_0>0$ such that for all $i\in \N$ we have
\[
\Inj(g_i)\geq i_0.
\]
\end{lemma}

\begin{proof}
By \th\ref{L: Non-collapsing dimension and lower curvature bound implies lower bound on volume of unit balls} we have that there exists a positive constant $\nu_0>0$ and $N\in\N$ such that for $i\geq N$, we have for all $p_i\in M_i$
\[
\mathrm{vol}(g_i)(B_1(p_i))\geq \nu_0.
\]
Thus for the sequence $\{(M_i,g_i)\}_{i\geq N}$, by \cite[Chapter~10, Section~4, Lemma~51]{Petersen2006} there exists a constant $i_1>0$ which depends only on $n$, $\max\{|\delta|, |\Delta|\}$ and $\nu_0$, such that for all $i\geq N$
\[
\Inj(g_i)\geq i_1. 
\]
We now  take $i_0 = \min\{\Inj(g_1),\ldots, \Inj(g_{N-1}), i_1\}>0$, to get the desired conclusion.
\end{proof}

With this, we obtain the following conclusion:

\begin{lemma}\th\label{L: Upper and lower curvature bounds imply that injectivity radius is positive}
Consider a compact complete length space $(X,d)$ with
\[
k\leq \curv(X)\leq K,
\]
and  let $\varepsilon>0$. Then for the sequence $\{g^\varepsilon_i\}_{i\in \N}$ converging to the $C^{1,\alpha}$-continuous Riemannian metric $g^\varepsilon$  on $X$ given by \th\ref{T: Approximating bounded spaces},  there exists a constant $\nu_0>0$ such that $\Inj(g^\varepsilon_i)\geq \nu_0$ for all $i$. 
\end{lemma}

\begin{proof}
It follows from  \th\ref{T: Approximating bounded spaces}  that $(X,d_{g^\varepsilon_i})$ converges to $(X, d)$ in the Gromov-Hausdorff sense. The conclusion then follows from  \th\ref{L: Bounded curvature and limit n-dimensional implies bound injectivity radius}.
\end{proof}

The following result states that a non-collapsing sequence of compact manifolds with uniform upper and lower bounds on sectional curvature, as well as a uniform bound on diameter, has a convergent subsequence in the $C^{1,\alpha}$-sense.

\begin{theorem}  \cite[Remark 1.4 and Theorem 1.7]{Peters87} (see also \cite[Theorem and Corollary on p. 121]{GW88})\th\label{T: Compactness in C alpha topology}
Let $0<\alpha<1$. Then any sequence $\{(M_i,g_i)\}$ of compact $n$-dimensional manifolds with $\diam(g_i)\leq D$ for $D>0$, $\Inj(g_i)\geq \nu_0>0$, and $|\Sec(g_i)|\leq K$ has a convergent subsequence converging in the $C^{1,\alpha}$-sense to a smooth $n$-dimensional manifold $M$ with $g$ a Riemannian metric of class $C^{1,\alpha}$.
\end{theorem}

\subsection{Limits of sequences of Ricci flows}

Here we recall some convergence results for sequences of Ricci flows under some geometric conditions.

Let $(M_i,g_i(t))$ be a sequence of smooth families of complete Riemannian manifolds for real $t\in(a,b)$ where $-\infty \leq a <0<b\leq\infty$, and suppose $p_i\in M_i$ for each $i$. Let $(M,g(t))$ be another smooth family of Riemannian manifolds for $t\in (a,b)$, and take $p\in M$. We say that \emph{$(M_i,g_i(t),p_i)$ converges in the Hamilton-Cheeger-Gromov sense to $(M,g(t),p)$} as $i\to \infty$ if there exist
\begin{enumerate}[(a)]
\item a sequence of compact subsets $\Omega_k\subset M$ which exhausts $M$ with $p\in \mathrm{int}(\Omega_k)$, 
\item a sequence of smooth maps $\phi_k\colon \Omega_k\to M_k$ which are a diffeomorpshism onto their image, with $\phi_k(p)=p_k$, and
\item the pullback metrics $\phi_k^\ast g_k(t)$ converge to $g(t)$ locally on $M\times (a,b)$ in the $C^\infty$-topology.
\end{enumerate}

\begin{rmk}
Observe that for a compact manifold $M$, in the definition of Hamilton-Cheeger-Gromov convergence, for sufficiently large $k$ we may assume that $\Omega_k= M$.
\end{rmk}

\begin{rmk}\th\label{R: Hamilton-Cheeger-Gromov convergence implies Cheeger-Gromov convergence}
Observe that for a sequence $\{(M_i,g_i(\cdot),p_i\}_{i\in \N}$ of compact manifolds converging in the Hamilton-Cheeger-Gromov sense to a compact manifold$(M,g(\cdot),p)$, we have for fixed $t\in (a,b)$ that $\{(M_i,g_i(t),p_i)\}_{i\in \N}$ converges to $(M,g(t),p)$ in the Cheeger-Gromov sense.
\end{rmk}

\begin{definition}
Let $(M,g)$ be an $n$-dimensional Riemannian manifold. Given an $(s,t)$-tensor $T$, we define the \emph{Frobenius norm of $T$, denoted by  $\|T\|$,} as follows: Let $(x_1,\ldots,x_n)$ be local coordinates over an open neighborhood $U\subset M$ of $p$. For $q\in U$ we set
\[
\|T\|^2(q) := \left(\sum g^{i_1 j_1}\cdots g^{i_t j_t}g_{k_1\ell_1}\cdots g_{k_s\ell_s} T_{i_1,\ldots,i_t}^{k_1,\ldots,k_s}T_{j_1,\ldots,j_t}^{\ell_1,\ldots,\ell_s}\right)(q).
\]
This definition is independent of the choice of coordinates.
\end{definition}

The following result gives sufficient geometric conditions that guarantee the existence of a limit flow for a family of flows satisfying equation \eqref{EQ: Ricci flow equation} under Cheeger-Gromov convergence. We recall that $\Rm(g)$ is the Riemann curvature covariant $4$-tensor field.

\begin{theorem}[Hamilton's Compactness Main Theorem~1.2 in \cite{Hamilton1995} and Theorem 1.1 in \cite{Topping2014}]\th\label{T: Hamiltons compactness theorem}
Let $\{(M_i,g_i(t))\}_{i\in \N}$  be a sequence of complete solutions to \eqref{EQ: Ricci flow equation} on $n$-dimensional manifolds defined on a non-empty interval $(a,b)$, with $a<0<b$. Furthermore assume that for $p_i \in M_i$ the following conditions hold:
\begin{enumerate}[(i)]
\item There exists $C>0$ constant such that $\|\Rm(g_i(t))\| <C$ for all $t\in (a,b)$ and all $i\in \N$.
\item There exists a constant $\nu_0>0$ such that for each index we have $\Inj(g_i(0)) \geq \nu_0$. 
\end{enumerate}
Then there exists a manifold $M$ of dimension $n$,  a complete solution to the Ricci flow $g_\infty(t)$ for $t\in (a,b)$, and $p\in M$ such that  after passing to a subsequence  we have that $(M_i,g_i(t)),p_i)$ converges to $(M,g_\infty(t),p)$ in the Hamilton-Cheeger-Gromov-sense.
\end{theorem}

\begin{rmk}\th\label{R: curvature bounds pass to the limit}
Observe that if for the sequence $\{(M_i,g_i(t))\}_{i\in \N}$ we have a curvature bound
\[
\delta\leq \Sec(g_i(t))\leq \Delta,
\]
then by the nature of the convergence in the Cheeger-Gromov sense  we obtain that the limit $(M,g(t))$ has $\delta\leq \Sec(g(t))\leq \Delta$.
\end{rmk}

In the case when the Riemann  curvature tensor $\Rm(g(t))$ of the Ricci flow has bounded Frobenius norm, by following the proof of \cite[Theorem 1.2]{Shi1989} we obtain the following proposition.

\begin{prop}\th\label{P: bounded Rm tensor implies metric equivalence}
Let $g(t)$ be the Ricci flow on a compact smooth $n$-dimensional manifold $M$ defined over $[0,T]$. Assume that there  exists $C_0>0$ such that
\[
\|\Rm(g(t))\|\leq C_0\quad 0\leq t\leq T.
\]
Then for all $t\in [0,T]$ we have
\[
e^{-2n\sqrt{C_0}t} g(0)\leq g(t)\leq e^{2n\sqrt{C_0}t} g(0).
\]
\end{prop}

Combining \th\ref{P: bounded Rm tensor implies metric equivalence} with \th\ref{T: Hamiltons compactness theorem} we obtain the following result. This is essentially the same as Hamilton's Compactness Theorem (\th\ref{T: Hamiltons compactness theorem}), with the distinction that the existence time of the Ricci flow is
$[0, b]$ instead of $(a, b)$ with $a<0<b$.
\begin{theorem}\th\label{T: Modified Hamiltons compactness theorem}
Let $\{(M_i,g_i(t))\}_{i\in \N}$  be a sequence of complete solutions to \eqref{EQ: Ricci flow equation}, on $n$-dimensional manifolds defined on a non-empty interval $[0,b]$. Furthermore assume that for $p_i \in M_i$ the following conditions hold:
\begin{enumerate}[(i)]
\item There exists $C>0$ constant such that $\|\Rm(g_i(t))\| <C$ for all $t\in [0,b]$ and all $i\in \N$.
\item There exists a constant $\nu_0>0$ such that for each index we have $\Inj(g_i(0)) \geq \nu_0$.
\end{enumerate}
Then there exists a manifold $M$ of dimension $n$,  a complete solution to the Ricci flow $g_\infty(t)$ for $t\in (0,b)$, and $p\in M$ such that  after passing to a subsequence  we have that $(M_i,g_i(t)),p_i)$ converges to $(M,g_\infty(t),p)$ in the Hamilton-Cheeger-Gromov-sense.
\end{theorem}

\begin{proof}
We show that there exists $0<t_0<b$ such that the hypothesis of \th\ref{T: Hamiltons compactness theorem} hold for the sequence of flows $\tilde{g}_i(t) = g_i(t_0+t)$ for $t\in (-t_0,b-t_0)$.

Since $\|\Rm(g_i(t))\| \leq C$ for all $t\in [0,b]$, then by \th\ref{P: bounded Rm tensor implies metric equivalence} there exists a constant $K:= n\sqrt{C}\geq 0$ such that for all $i\in \N$ we have for $t\in [0,b]$ we have
\begin{linenomath}
\begin{align}\label{EQ: metric equivalence}
    e^{-2Kt}g_i(0)\leq g_i(t)\leq e^{2Kt}g_i(0).
\end{align}
\end{linenomath}

Fix $p\in M$, and consider two orthonormal vectors $u,v\in T_p M$. Let $\{e_1,e_2,\ldots,e_n\}$ be an orthonormal basis of $T_p M$ with $e_1=u$, $e_2=v$. Then we have that
\begin{linenomath}
\begin{align*}
|\Sec(g_i(t))(u\wedge v)|^2 &= (\Sec(g_i(t))(u\wedge v))^2\\ 
&= (\Rm(g_i(t))(u,v,u,v))^2\\ 
&= \Rm(g_i(t))^2_{1212}(p)\\
&\leq \sum_{i,j,k,\ell=1}^n \Rm(g_i(t))^2_{ijk\ell}(p)\\ 
&=\|\Rm(g_i(t))\|^2(p)\leq C^2.
\end{align*}
\end{linenomath}
Thus there exists $\Delta>0$ such that $|\Sec(g_i(t))|\leq \Delta$ for all $i\in \N$ and $t\in[0,b]$. 

Fix $t_0\in (0,b)$, and consider a $C^1$-continuous curve $\gamma_i\colon [0,1]\to M_i$. We have the following relation between the length of $\gamma_i$ with respect to $g_i(0)$, denoted by $\ell(g_i(0))(\gamma_i)$, and the length of $\gamma_i$ with respect to $g_i(0)$, denoted by $\ell(g_i(t_0))(\gamma_i)$:
\begin{linenomath}
\begin{align*}
\ell(g_i(0))(\gamma_i) =& \int_0^1 \sqrt{g_i(0)(\gamma'(s),\gamma'(s))}\, ds\\
\geq& \int_0^1 \sqrt{e^{-2Kt_0} g_i(t_0)(\gamma'(s),\gamma'(s))}\, ds\\
\geq& e^{-Kt_0}\int_0^1 \sqrt{g_i(t_0)(\gamma'(s),\gamma'(s))}\, ds\\
&= e^{-Kt_0}\ell(g_i(t_0))(\gamma_i).
\end{align*}
\end{linenomath}
From this, it follows that for all $i\in \N$  and arbitrary $x_i,y_i\in M_i$ we have
\[
d_{g_i(0)}(x_i,y_i)\geq e^{-Kt_0}d_{g_i(t_0)}(x_i,y_i).
\]
For $p_i$ fixed and $r>0$, let $q_i\in B_{e^{-Kt_0}r}^{g_i(0)}(p_i)$. Then since $e^{-Kt_0}r> d_{g_i(0)}(p_i,q_i)$, it follows that
\[
r> e^{Kt_0}d_{g_i(0)}(p_i,q_i)\geq d_{g_i(t_0)}(p_i,q_i), 
\]
That is the open ball $B_{e^{-Kt_0}r}^{g_i(0)}(p_i)$ of radius $e^{-Kt_0}r$ centered at $p_i$ with respect to $g_i(0)$ is contained in the open ball of radius $r$ centered at $p_i$ with respect to $g_i(t_0)$, i.e.
\[
B_{e^{-Kt_0}r}^{g_i(0)}(p_i)\subset B_r^{g_i(t_0)}(p_i). 
\]

Observe that since $t_0\in (0,b)$ and $K>0$, then we have $1\geq e^{-Kt_0}\geq e^{-Kb}$, and $1\geq e^{-Knt_0}\geq e^{-Knb}$. Choose $2r_0 =\min\{\nu_0,2\}$. Then we have that $2e^{-Kt_0}r_0\leq 2r_0\leq \nu_0 \leq \Inj(g_i(0))$ for all $i\in \N$. Then by \cite[Proposition 14]{Croke1980}, there exists a constant $A(n)$, which depends only on the dimension $n$, such that:
\[
\mathrm{vol}(g_i(0))\Big(B_{e^{-Kt_0}r_0}^{g_i(0)}(x_i)\Big)\geq \frac{A(n)}{n^n}r_0^ne^{-Knt_0}.
\]
for each $i\in \N$ and all $x_i\in M_i$. 

Thus, since $r_0\leq 1$ by construction, we have for  arbitrary $x_i\in M_i$ that
\begin{linenomath}
\begin{align*}
\mathrm{vol}(g_i(t_0))\big(B_{1}^{g_i(t_0)}(x_i)\big)&\geq \mathrm{vol}(g_i(t_0))\big(B_{r_0}^{g_i(t_0)}(x_i)\big)\\
&\geq \mathrm{vol}(g_i(t_0))\big(B_{e^{-Kt_0}r_0}^{g_i(0)}(x_i)\big)\\
&\geq e^{-Knt_0}\mathrm{vol}(g_i(0)) \big(B_{e^{-Kt_0}r_0}^{g_i(0)}(x_i)\big) \quad\mbox{ (by } \eqref{EQ: metric equivalence}\mbox{)}\\ % Formula for rescaled metrics.
&\geq e^{-Knt_0}\frac{A(n)}{n^n}r_0^ne^{-Knt_0}\\
&\geq \frac{A(n)}{n^n}r_0^n e^{-2Knb}.
\end{align*}
\end{linenomath}
We set 
\[
\tilde{\nu}_0=\tilde{\nu}(C,n,b,\nu_0):= \frac{A(n)}{n^n}r_0^n e^{-2Knb} = \frac{A(n)}{n^n}\left(\min\left\{\frac{\nu_0}{2},1\right\}\right)^n e^{-2Knb}>0.
\]
Observe that since $2r_0 =\min\{\nu_0,2\}\leq \Inj(g_i(0))$ for all $i\in \N$,  and $e^{-2Knb}\leq 1$ we have that for all $i\in \N$:
\begin{linenomath}
\begin{align*}
\mathrm{vol}(g_i(0))\big(B_1^{g_{i}(0)}(x_i)\big) &\geq \mathrm{vol}(g_i(0))\big(B_{r_0}^{g_{i}(0)}(x_i)\big)\\
& \geq \frac{A(n)}{n^n}r_0^n\\
&\geq \frac{A(n)}{n^n} \left(\min\left\{\frac{\nu_0}{2},1\right\}\right)^n e^{-2Knb}.
\end{align*}
\end{linenomath}

Thus for all $i\in\N$ and all $t_0\in [0,b)$, we have
\begin{equation}\label{EQ: lower uniform bound for volume of unit ball}
\mathrm{vol}(g_i(t_0))\big(B_{1}^{g_i(t_0)}(x_i)\big)\geq \tilde{\nu}_0>0.
\end{equation}
Then by \cite[Chapter~10, Section~4, Lemma 51]{Petersen2006}, there exists a constant $\iota_0 = \iota(\tilde{\nu}_0,\Delta,n)>0$ such that
\[
\Inj(g_i(t_0))\geq \iota_0,
\]
for all indices $i\in \N$ and all times $t_0\in [0,b)$. (observe that this lower bound $\iota_0$ may not be the same as the lower bound $\nu_0$).

Thus for a fixed $t_0\in (0,b)$, we have that the injectivity radius of $g_i(t_0)$ is uniformly bounded by a positive constant, and thus the injectivity radius of $\tilde{g}_i:= g_i(t_0)$ is uniformly bounded from below by a positive constant. Moreover, we observe  that $\|\Rm(\tilde{g}_i(t))\| = \|\Rm(g_i(t_0+t)\|\leq C$ for all $t\in (-t_0,b-t_0)$. This implies that the sequence of Ricci flows $\tilde{g}_i(t)$ with initial conditions $\tilde{g}_i(0) = g_i(t_0) = \tilde{g}_i$ satisfy the hypothesis of \th\ref{T: Hamiltons compactness theorem}. Thus there exists a smooth manifold $M$ and a Ricci flow $\tilde{g}_\infty(t)$ defined on $M$ for $t\in (-t_0,b-t_0)$. Moreover, given $p\in M$ and $t\in(t_0,b-t_0)$ we have that after passing to a subsequence, we have that $(M_i,\tilde{g}_i(t)= g_i(t+t_0),p_i)$ converges to $(M,\tilde{g}_\infty(t),p)$ in the Hamilton-Cheeger-Gromov-sense. For $t\in (0,b)$ we set up $g_\infty(t):=\tilde{g}_\infty(t-t_0)$. This is the desired Ricci flow on $M$ defined over $(0,b)$ for which the conclusions hold.
\end{proof}

\subsection{Lower bounds on the maximal existence time for the Ricci flow}

In his foundational study of the Ricci flow, Hamilton derived an estimate that provides a lower bound for the maximal existence time of the solution to \eqref{EQ: Ricci flow equation} as follows .

\begin{lemma}[Doubling Lemma, Lemma~7.6 in \cite{AndrewsHopper11}]\th\label{L: Doubling Time Estimate}
Let $(M,g)$ be a compact $n$-dimensional manifold with $\|\Rm(g)\|\leq K$. Then for a solution $g(t)$ to \eqref{EQ: Ricci flow equation} with $g(0)=g$,  there exists a constant $C(n)$, depending only on the dimension, such that for $t\in [0,C(n)/K]$ the solution $g(t)$ is defined and we have 
\[
	\|\Rm(g(t))\|\leqslant 2K.
\]
\end{lemma}

We also have the following relation between the sectional curvature of a manifold, and the coefficients of the Riemann curvature  tensor.

\begin{prop}[Proposition 2.49 in \cite{AndrewsHopper11}]
Let $(M,g)$ be a Riemannian manifold. The Riemannian curvature tensor $\Rm(g)$ is completely determined by the sectional curvature $K =\Sec(g)$ of $(M,g)$. Moreover, if $\{e_1,\ldots,e_n\}$ is an orthonormal basis of $T_p M$, we have that
\begin{linenomath}
\begin{align*}
  \Rm(e_i,e_j,e_k,e_\ell)(p) =  \frac{1}{3}&K\left(\frac{(e_i+e_k )\wedge (e_j+e_\ell)}{2} \right)+\frac{1}{3}K\left(\frac{(e_i-e_k) \wedge (e_j-e_\ell)}{2} \right)\\
  &-\frac{1}{3}K\left(\frac{(e_j+e_k) \wedge (e_i+e_\ell)}{2} \right)-\frac{1}{3}K\left(\frac{(e_j-e_k) \wedge (e_i-e_\ell)}{2} \right)\\
  &-\frac{1}{6}K\left(e_j\wedge e_\ell \right)-\frac{1}{6}K\left(e_i\wedge e_k \right)+\frac{1}{6}K\left(e_i\wedge e_\ell \right)+\frac{1}{6}K\left(e_j\wedge e_k \right).
\end{align*}
\end{linenomath}
\end{prop}

From this we get the following lemma:

\begin{lemma}\th\label{L: Sectional curvature bounds imply curvature bound on Riemannian tensor}
Let $(M,g)$ be an $n$-dimensional Riemannian manifold with $\delta\leq \Sec(g)\leq \Delta$. Then we have that $\|\Rm(g)\|\leq n^2(\Delta-\delta)$.
\end{lemma}

\begin{proof}
We fix $p\in M$ and consider $(x_1,\ldots,x_n)$-normal coordinates of $M$ centered at $p$. Then we have that $\{e_1=\partial_1(p),\ldots,e_n=\partial_n(p)\}$ form an orthnormal basis at $T_p M$. Thus we get
\[
\|\Rm(g)\|^2(p) = \sum_{i,j,k,\ell=1}^n \Rm(g)^2_{i,j,k,\ell}(p).
\]
We also get that 
\begin{linenomath} 
\begin{align*}
  \Rm_{i,j,k,\ell}(p) =  \frac{1}{3}&K\left(\frac{(e_i+e_k)\wedge (e_j+e_\ell)}{2} \right)+\frac{1}{3}K\left(\frac{(e_i-e_k) \wedge (e_j-e_\ell)}{2} \right)\\
  &-\frac{1}{3}K\left(\frac{(e_j+e_k) \wedge (e_i+e_\ell)}{2} \right)-\frac{1}{3}K\left(\frac{(e_j-e_k)\wedge (e_i-e_\ell)}{2} \right)\\
  &-\frac{1}{6}K\left(e_j\wedge e_\ell \right)-\frac{1}{6}K\left(e_i\wedge e_k \right)+\frac{1}{6}K\left(e_i\wedge e_\ell \right)+\frac{1}{6}K\left(e_j\wedge e_k \right)\\
  &\leq \Delta-\delta.
\end{align*}
\end{linenomath}
In analogous fashion we obtain $\Rm_{i,j,k,\ell}(p)\geq \delta-\Delta$. From this it follows that $|\Rm_{i,j,k,\ell}(p)|\leq (\Delta-\delta)$, and thus
\[
\|\Rm(g)\|^2(p) = \sum_{i,j,k,\ell=1}^n \Rm(g)^2_{i,j,k,\ell}(p)
= \sum_{i,j,k,\ell=1}^n |\Rm_{i,j,k,\ell}(p)|^2\leq \sum_{i,j,k,\ell=1}^n  (\Delta-\delta)^2 = n^4(\Delta-\delta)^2, \]
and our conclusion follows.
\end{proof}
\begin{rmk}
In the case when $\delta \geq 0$, by \cite[Proposition 2.50]{AndrewsHopper11} we have a tighter bound on $\|\Rm\|$.
\end{rmk}

\subsection{Convergence of limit flow}\label{S: convergenece of limit flow}

For the sake of completeness, we present the following results from \cite{SimonTopping2021} and \cite{DeruelleSchulzeSimon2022}, which will be used to guarantee that a limit Ricci flow defined over $(0,T)$ has weak limit as $t\to 0$.

\begin{theorem}[Lemma 3.1 in \cite{SimonTopping2021} and Theorems~1.3 and 1.6 in \cite{DeruelleSchulzeSimon2022}]\th\label{T: Existence of limit}
Let $(M,g(t))$ be a smooth solution to \eqref{EQ: Ricci flow} for $t\in (0,T]$,  not necessarily complete,  with the property that for some $x_0\in M$ and $r>0$, and all $0<t\leq T$,  the open ball $B^{g(t)}_{2r}(x_0)$ has compact closure in $M$. Furthermore,  suppose that for some $c_0,\alpha>0$ and for each $t\in (0,T]$ we have
\begin{linenomath}
\begin{align}
\|\Rm(g(t))\|\leq \frac{c_0}{t},
\end{align}
\end{linenomath}
and 
\begin{linenomath}
\begin{align}
\Ric(g(t))\geq -\alpha.
\end{align}
\end{linenomath}
on $B_{2r}^{g(t)}(x_0)$. 
Then the following statements hold:
\begin{enumerate}[(i)]
\item On $\Omega_T = \cap_{t\in (0,T]} B^{g(t)}_r(x_0)$ we have that the metrics $d_{g(t)}$ converge uniformly to a metric $d_0$ and for all $0<t\leq T$:
\begin{linenomath}
\begin{align}
e^{\alpha t}d_0\geq d_{g(t)}\geq d_0-\beta(n) \sqrt{c_0 t},\label{E: convergence rate limit flow}
\end{align}
\end{linenomath}
where $\beta(n)$ is some constant depending on the dimension $n$ of $M$.
\item There exists $R_0>0$ such that for all $t\in (0,T]$ we have the open ball $B^{g(t)}_{R_0}(x_0)\subset \Omega_T$, and for $0<R<R_0$ the topology of $B^{d_0}_R(x_0)$ induced by $d_0$ agrees with  the subspace topology  induced by $M$.
\item In the case that the open ball  $(B^{d_0}_R(x_0),d_0)$, $R\leq R_0$, is isometric to a Riemannian manifold, then there exists a Riemannian metric $g_0$ on $B^{d_0}_s(x_0)$ for some $s\in (0,R)$ such that we can extend the Ricci flow solution $(B^{d_0}_s(x_0),g(t))$ for $t\in (0,T]$ to a smooth solution $(B^{d_0}_s(x_0),g(t))$ for $t\in [0,T]$ by setting $g(0) = g_0$.
\end{enumerate}
\end{theorem}

\section{Proof of main results}\label{S: Proof of main theorem}

In this section, we prove our main results. We begin by proving Theorem~\ref{MT: Main Theorem 1}. First we prove Theorems~\ref{T: Machinery} and \ref{T: Uniqueness}. \th\ref{T: Machinery} is needed to prove the existence of the flow, and  \th\ref{T: Uniqueness} proves the uniqueness.

Let $(X,d)$ be a complete length space which is the Gromov-Hausdorff limit of a sequence of closed smooth Riemannian manifolds $\{(M_i,g_i)\}_{i\in \N}$.
For each index $i\in \N$, we consider the Riemannian metric $g_i(t)$ which is a solution to the Ricci flow equation
\begin{linenomath}
\begin{equation}\label{EQ: Ricci flow equation}
	\frac{\partial g_i}{\partial t}(t) = -2\Ric\Big(g_i(t)\Big)\tag{\mbox{Ricci}},
\end{equation}
\end{linenomath}
with initial condition $g_i(0) = g_i$.

\begin{theorem}\th\label{T: Machinery}
Consider a complete length space  $(X,d)$ which is the Gromov-Hausdorff limit of a sequence of  closed smooth Riemannian manifolds $\{(M_i,g_i)\}_{i\in \N}$. Assume further that 
\begin{enumerate}
\item\label{T: Machinery (i)} There exists a constant $K_0>0$ such that for each index $i$ we have 
\[
K_i = \sup_{M_i} \|\Rm(g_i)\|\leq K_0.
\]
\item\label{T: Machinery (ii)} There exists a constant $\nu_0>0$ such that for each index $\Inj(g_i) \geq \nu_0$.
\end{enumerate}
Then there exists a smooth complete solution $(M,g_\infty(t))$ to \eqref{EQ: Ricci flow equation} on an interval $(0,C/K_0)$, where $C$ is a constant depending only on the dimension of $X$. Moreover, $X$ is a smooth manifold, diffeomorphic to $M$, and  there exist a  $C^{1,\alpha}$-continuous Riemannian metric $g_0$ on $X$ with $d_{g_0}$ isometric to  $d$, such that $(M,g_\infty(t))$ converges in the $C^{1,\alpha}$-topology to $g_0$ as $t\to 0$.
\end{theorem}

\begin{proof}
First note that by Condition~\ref{T: Machinery (ii)}, $(M_i, g_i)$ in a non-collapsing sequence and hence $\dim X=\dim~M_i=n$, for all $i$.   We consider solutions $\{(M_i,g_i(t))\}$ to  \eqref{EQ: Ricci flow equation} with initial conditions $g_i(0) = g_i$. By \th\ref{L: Doubling Time Estimate}, these flows exist on the interval $[0,C/K_0]$ with non-empty interior for all $i\in \N$, where $C$ is  a constant depending only on $n$. 

Moreover, on the interval $[0,C/K_0]$, the flows $g_i(t)$ have bounded Riemannian curvature tensor as seen by the Doubling Time Estimate Lemma (\th\ref{L: Doubling Time Estimate}). That is for $t\in[0,C/K_0]$ it holds
\[
	\|\Rm(g_i(t))\|\leqslant 2K_0.
\]

Thus the hypotheses of \th\ref{T: Modified Hamiltons compactness theorem} are satisfied. This implies that for a selection of base points $p_i\in M_i$, there exists a smooth manifold $M$ equipped with a complete metric $g_\infty(t)$, with $t\in (0, C/K_0)$, which satisfies equation \eqref{EQ: Ricci flow equation}, and there exists $p\in M$ such that $(M_i,g_i(t),p_i)\to (M,g_\infty(t),p)$ in the Hamilton-Cheeger-Gromov sense. Moreover, for fixed $t\in (0,C/K_0)$ we have  $(M_i,g_i(t),p_i)\to (M,g_\infty(t),p)$ in the Cheeger-Gromov sense (see \th\ref{R: Hamilton-Cheeger-Gromov convergence implies Cheeger-Gromov convergence}).

Since for all $t\in [0,C/K_0]$ we have $\|\Rm(g_i(t))\|\leq 2K_0$ for $K_0>0$, we have that $-\overline{K}\leq \Ric(g_i(t))$ for some $\overline{K}>0$. Note that for $0<t<C/K_0$, we have that $2K_0<2C/t$. That is, 
\[
\|\Rm(g_i(t))\|\leq 2K_0<\frac{2C}{t}.
\]
Then  since $(M_i,g_i(t))$ converges in the Cheeger-Gromov sense to $(M,g_\infty(t))$ for all $t\in (0,C/K_0)$, the upper bounds on $\Rm_{ijk\ell}^2(g_i(t))$ are preserved in the limit $g_\infty(t)$ for all $t\in (0,C/K_0)$ (see for example \cite[Proof of Theorem 1.6]{Topping2014} or \cite[Page 2236]{DeruelleSchulzeSimon2022} together with the bound \eqref{EQ: lower uniform bound for volume of unit ball}). Thus,  for all $t\in (0, C/K_0)$ we have 
\[
\|\Rm(g_\infty(t))\|\leq2K_0< \frac{2C}{t},\quad \mbox{and} \quad -\overline{K}\leq \Ric(g_\infty(t)). 
\]
This implies that for a fixed $x_0\in M$, by \th\ref{T: Existence of limit} there exist a compact neighborhood $U$  of $x_0$ in $M$, and a metric $d_0$ on $U$ such that on $U$ the metrics $d_{g_\infty(t)}$ converge uniformly to the metric $d_0$. 

Consider another compact neighborhood $U'$ of a point $x'$ in $M$ given by \th\ref{T: Existence of limit}, equipped with the metric $d_0'$, such that $U \cap U' \neq \emptyset$.  By Theorem~\ref{T: Existence of limit}, for $x, y \in U \cap U'$, we have $d_{g_\infty(t)}(x, y) \to d_0(x, y)$ due to the uniform convergence over $U$, and $d_{g_\infty(t)}(x, y) \to d_0'(x, y)$ due to the uniform convergence over $U'$ as $t \to 0$. This implies that the distance functions $d_0$ and $d_0'$ must agree on $U \cap U'$. Consequently, there exists a well-defined distance function $d_0$ on $M$ such that $d_{g_\infty(t)}$ converges locally uniformly to $d_0$. Since $M$ is compact, this local uniform convergence implies that $d_{g_\infty(t)}$ converges to $d_0$ uniformly on $M$ as $t \to 0$. Therefore, $(M, d_{g_\infty(t)})$ converges to $(M, d_0)$ in the Gromov-Hausdorff sense.

Now, we prove  that $(M,d_0)$ is isometric to $(X,d)$. First observe that by the definition of the Hamilton-Cheeger-Gromov-convergence, the fact that each manifold $M_i$ in the sequence and $M$ are compact, and \th\ref{R: Approximation for spaces of bounded curvature is also Gromov-Hausdorff} the following statement holds:  for $\varepsilon>0$, there exists $N_1\in \N$ such that for all $t\in (0,C/K_0)$ and $i\geq N_1$ we have
\[
d_{\mathrm{GH}}((M_i,d_{g_i(t)}),(M,d_{g_\infty(t)}))< \varepsilon.
\]
Moreover since $(M_i,d_{g_i(0)})$ converges to $(X,d)$ in the Gromov-Hausdorff sense, there exists $N_2\in \N$ such that for $i\geqslant N_2$ we have
\[
d_{\mathrm{GH}}((M_i,d_{g_i(0)}),(X,d))< \varepsilon.
\]

Let $N=\max\{N_1, N_2\}$. Since the Ricci flow is smooth with respect to time, for the flow $(M_N, g_N(t))$ there exists $t_N$ such that for all $t\leq t_N$, 
\[
d_{\mathrm{GH}}((M_N,d_{g_N(0)}),(M_N, d_{g_N(t)}))< \varepsilon.
\]
Then  by construction, for $t\in(0, t_N)$ we have
\begin{linenomath}
\begin{align*}
d_{\mathrm{GH}}((M,d_{g_\infty(t)}),(X,d))\leq& d_{\mathrm{GH}}((M_N,d_{g_N(t)}),(M,d_{g_\infty(t)})) + d_{\mathrm{GH}}((M_N,d_{g_N(0)}),(M_N,d_{g_N(t)}))\\ 
&+  d_{\mathrm{GH}}((M_N,d_{g_N(0)}),(X,d)) < 3\varepsilon.
\end{align*}
\end{linenomath}

This implies that $(M,d_{g_\infty(t)})$ converges to $(X,d)$ in the Gromov-Hausdorff sense as $t\to 0$. Thus by uniqueness of limits we conclude that $(X,d)$ is isometric to $(M,d_0)$.

We prove now that for any $0<\alpha<1$ the metric $d_0$ on $M$ is, up to an isometry, induced by a $C^{1,\alpha}$-continuous Riemannian metric, and that $(M,g_\infty(t))$ converges to $(M,g)$ in the $C^{1,\alpha}$-sense. Since we have that $\Ric(g_\infty(t))\geq -\overline{K}$, then by \eqref{E: convergence rate limit flow} we have for any $x,y\in M$ and $t\in (0, C/K_0)$ that
\[
e^{\overline{K}(C/K_0)}\mathrm{Diam}(M,d_0)\geq e^{\overline{K}t}d_0(x,y)\geq d_{g_\infty(t)}(x,y).
\]
From this we conclude that $\mathrm{Diam}(M,d_{g_\infty(t)})\leq e^{\overline{K}(C/K_0)}\mathrm{Diam}(M,d_0)$ for $t\in (0,C/K_0)$.

Let $\{t_j\}_{j\in \N}$ be a sequence in $(0, C/K_0)$ converging to $0$. Consider the sequence $(M, g_\infty(t_j))$. By our assumptions this sequence satisfies the conditions of Theorem~\ref{T: Compactness in C alpha topology}, therefore, a subsequence $(M,g_\infty(t_{j_k}))$ converges in the $C^{1,\alpha}$-topology to  a $C^{1,\alpha}$-continuous Riemannian manifold $(Y,g_Y)$. By \th\ref{R: Approximation for spaces of bounded curvature is also Gromov-Hausdorff} it follows that $(M,d_{g_\infty(t_{j_k})})$ converges in the Gromov-Hausdorff topology to $(Y,d_{g_Y})$. By the uniqueness of the limits we have that $(Y,d_{g_Y})$ is isometric to $(M,d_0)$. Thus we have that $(M,g_\infty(t))$ converges to a space isometric to $(X,d)$ in the $C^{1,\alpha}$-topology. 
\end{proof}

Now we adapt the arguments in \cite[Section 5]{BurkhardtGuim2019} to show that the limit flow is independent of the sequence used to approximate $(X,d)$.

\begin{theorem}\th\label{T: Uniqueness}
Let $\{(N_i,h_i)\}_{i\in \N}$ be another  sequence  which  converges to $(X,d)$ in the Gromov-Hausdorff sense and satisfies the hypotheses of \th\ref{T: Machinery} for possibly a different positive uniform lower bound $\mu$ for the injectivity radii. Let $(N_\infty,h_\infty(t))$ be the Ricci flow given by \th\ref{T: Machinery} induced by the sequence. Then there exists a diffeomorphism $\alpha\colon M_\infty\to N_\infty$ such that for any $t\in (0,C/K_0)$ we have $\alpha\colon  (M_\infty,g_\infty(t))\to (N_\infty,h_\infty(t))$ is an isometry.
\end{theorem}

\begin{proof} 
Assume that $\{(N,h_i)\}$ is another sequence of $n$-dimensional closed Riemannian manifolds converging to $(X,d)$ in the Gromov-Hausdorff sense such that  $\delta\leq \Sec(h_i)\leq \Delta$ and $\Inj(h_i)\geq \mu>0$. Then we obtain a Ricci flow $(N,h_\infty(t))$ defined over $(0,C/K_0)$ the same interval as $(M,g_\infty(t))$ and a $C^{1,\alpha}$ Riemannian manifold $(\bar{Y},h)$ isometric to $(X,d)$ such that $(N_\infty,h_\infty(t))$ converges in $C^{1,\alpha}$-topology to $(\bar{Y},h)$ when $t\to 0$.

Observe that there are $C^{1, \alpha}$ isometries $\phi\colon (X, g)\to (Y, g_Y)$ and $\psi\colon (X,g)\to (\bar Y, h)$. Therefore, we have a $C^{1, \alpha}$ isometry $\theta\colon (Y, g_Y)\to (\bar Y, h)$.

Since we can choose a sequence $t_{i_j}\searrow 0$ such that we have (smooth) diffeomorphisms $\phi_{i_j}\colon Y\to M$  and $\psi_{i_j}\colon \bar Y\to N$, and $\phi_{i_j}^\ast(g_\infty(t_{i_j}))\to g_Y$ and $\psi_{i_j}^\ast(h_\infty(t_{i_j}))\to h$ in $C^{1, \alpha}$,  there exists a sequence of numbers $\varepsilon_{i_j}\searrow 0$ such that $\theta\colon (Y, \phi_{i_j}^\ast(g_\infty(t_{i_j})))\to (\bar Y, \psi_{i_j}^\ast(h_\infty(t_{i_j})))$ is $(1+\varepsilon_{i_j})$-bi-Lipschitz.

Now, from \cite[Section~4]{Karcher1977}, there exists a sequence of numbers $\delta_{i_j}\searrow 0$ and a  sequence of smooth functions     $\theta_{i_j}\colon (Y, \phi_{i_j}^\ast(g_\infty(t_{i_j})))\to (\bar Y, \psi_{i_j}^\ast(h_\infty(t_{i_j})))$  which are  $(1+\delta_{i_j})$-bi-Lipschitz. Moreover, $\theta_{i_j}\colon(Y, g_Y)\to (\bar Y, h)$  is  $(1+\delta_{i_j}+\gamma_{i_j})$-bi-Lipschitz, for all $i_j$, for some sequence $\gamma_{i_j}\searrow 0$. Hence by Arzelà-Ascoli, $\theta_{i_j}$ converges uniformly to a continuous map $\Theta$. Now, we follow the proof  of \cite[Lemma~5.8]{BurkhardtGuim2019} to conclude that there is a diffeomorphism $\alpha \colon M\to N$ such that  $\alpha^\ast (h_\infty(t))=g_\infty(t)$ for all $t\in (0, C/K_0)$.
\end{proof}

With \th\ref{T: Machinery} and \th\ref{T: Uniqueness} proven we almost have the proof of \th\ref{MT: Main Theorem 1}. We only have to prove that the hypothesis \th\ref{T: Machinery}~\eqref{T: Machinery (i)} is satisfied.

\begin{proof}[Proof of \th\ref{MT: Main Theorem 1}]
Consider $\{(M_n,g_n)\}_{n\in \N}$ the sequence of $m$-dimensional compact manifolds converging to $(X,d)$ in the Gromov-Hausdorff topology with 
\[
\delta\leq \mathrm{Sec}(g_n)\leq \Delta\quad \mbox{and}\quad \mathrm{Injrad}(g_n)\geq v>0.
\]
By \th\ref{L: Sectional curvature bounds imply curvature bound on Riemannian tensor} we have that for each index $i\in \N$ we have
\[
 \|\Rm(g_n)\|\leq m^2\left(\Delta-\delta\right)
\]
Setting $K = m^2(\Delta-\delta)$, by \th\ref{L: Doubling Time Estimate} we have that the solutions  $g_n$ of \eqref{EQ: Ricci flow equation} with initial conditions $g_n(0) = g_n$ are defined on the open interval $(0,C/K)$, where $C=C(m)$ is a constant depending only on the dimension of $M$. Moreover on this interval we have that  $\|\Rm(g_n(t))\|$ is bounded by $2K$ for all $t\in (0,C/K)$.

Thus the hypothesis of  \th\ref{T: Machinery} are satisfied for the sequence $\{g_n\}_{n\in \N}$. This implies that there exists a manifold $M$ and  a solution to \eqref{EQ: Ricci flow equation} $g_\infty(t)$ such that,  up to a subsequence, the flows  $g_n(t)$ solving \eqref{EQ: Ricci flow equation} with initial conditions $g_n(0) =g_n$ converge to $g_\infty(t)$ in the Hamilton-Cheeger-Gromov-sense. We also have that there exists a Riemannian manifold $(Y,g_Y)$ with a $C^{1, \alpha}$-continuous metric  such that  for every sequence $t_j\to 0$, there is   a subsequence $t_{j_k}$ such that the sequence  $\{(M,g_\infty(t_{j_k}))\}_{k\in \N}$ converges to $(Y,g_Y)$ in the $C^{1,\alpha}$-sense, and $(Y,d_{g_Y})$ is isometric to $(X,d)$ via an isometry $I\colon (Y,d_{g_Y})\to (X,d)$.

Moreover, since $(X,d)$ is both an Alexandrov space of curvature $\geq \delta$ and locally a $\mathsf{CAT}(\Delta)$-space, as pointed out in \cite[p. 221, Remark 2]{BerestovskijNikolaev1993} $X$ is a smooth manifold and the metric $d$ is induced by a $C^{1,\alpha}$-continuous Riemannian metric $g$. Then by \cite[Theorem 2.1]{Taylor2006} the isometry $I\colon (Y,d_{g_Y})\to (X,d)$ is $C^{2}$-continuous. Thus we have that $I_\ast(g_Y)=g$. We consider $\phi_k$ the sequence  diffeomorphisms onto its images $\phi_k\colon Y\to M_k$ realizing the $C^{1,\alpha}$-convergece. Then the sequence $\phi_k\circ I^{-1}$ gives us that the sequence  $\{(M,g_\infty(t_{j_k}))\}_{k\in \N}$ converges in the $C^{1,\alpha}$-sense to $(X,g)$. Since the sequence ${t_{j}}_{j\in \N}$ with $t_{j}\to 0$ is arbitrary, we conclude that that $(M,g_\infty(t))$ converges to $(X,g)$ in the $C^{1,\alpha}$-sense as $t\to 0$.

By \th\ref{T: Uniqueness}, we have that up to isometry, the flow is unique.
\end{proof}

We conclude this section  by providing the proofs of \th\ref{MT: Main Theorem} and \th\ref{MT: Equivariant}.

\begin{proof}[Proof of \th\ref{MT: Main Theorem}]
Fix $\varepsilon>0$. By \th\ref{T: Approximating bounded spaces}, there exists a sequence $\{(M,g^\varepsilon_i)\}_{i\in \N}$ of smooth Riemannian metrics converging to a  $C^{1,\alpha}$-continuous metric $g^\varepsilon$ on $X$, with $d_{g^\varepsilon} = d$.

Set $\delta_\varepsilon = k-\varepsilon$ and $\Delta_\varepsilon = K+\varepsilon$. Then we have 
\[
\delta_\varepsilon\leq \mathrm{Sec}(g_i^\varepsilon)\leq \Delta_\varepsilon.
\]

By \th\ref{L: Upper and lower curvature bounds imply that injectivity radius is positive} we have that $\Inj(g^\varepsilon_i)\geq \nu_0$ for some $\nu_0>0$.

Thus the hypotheses of \th\ref{MT: Main Theorem 1} are satisfied and our conclusion follows. 
\end{proof}

\begin{proof}[Proof of \th\ref{MT: Equivariant}] 
Let $\varepsilon>0$ be fixed. 
By \cite[Theorem D]{Ahumada2023},  we can assume that for the sequence of Riemannian metrics $g_n=g_n^\varepsilon$ on $X=M$  given by \th\ref{T: Approximating bounded spaces}, we have $\delta-\varepsilon\leq \Sec(g^\varepsilon_n)\leq \Delta+\varepsilon$, and $G<\mathrm{Iso}(M,g_n^\varepsilon)$. Moreover, by the uniqueness of the Ricci flow with respect to smooth initial conditions, we have $G<\mathrm{Iso}(M_n,g_n(t))$. Now let $(M_\infty,g_\infty(t))$ be the solution to the Ricci flow given by \th\ref{MT: Main Theorem 1}. Let $K_0 = \dim(X)^2(\Delta-\delta)$, and $C=C(\dim(X))$ be the constant from \th\ref{L: Doubling Time Estimate}. Then for $t_0\in (0,C/K_0)$ fixed, we have that $(M,g_n(t_0))$ converges to $(M_\infty,g_\infty(t_0))$ in the Gromov-Hausdorff sense. Then  by \cite[Proposition 3.6]{FukayaYamaguchi1992}, there exists a closed subgroup $\Gamma_{t_0}$ of $\mathrm{Iso}(M_\infty,g_\infty(t_0))$. Since $\mathrm{Iso}(M_\infty,g_\infty(t_0))$ is a Riemannian manifold then by \cite{MyersSteenrod1939}, its group of isometries is a Lie group, and thus $\Gamma_{t_0}$ is a Lie group. Moreover, by \th\ref{R: curvature bounds pass to the limit},  we have that $\Sec(g_\infty(t_0))\geq \delta-\varepsilon$. Further, since the sequence $(M,g_n(t_0))$ is non-collapsed, by \cite[Proposition 4.1]{Harvey2016}, there also exists injective Lie group morphism $\bar{\theta}_{n,t_0}\colon G\to \Gamma_{t_0}$.

As in the proof of \th\ref{T: Machinery}, we observe that given $\delta>0$ there exists $N\in \N$ and $t_N\in (0,C/K_0)$ such that for any $t\in (0,t_N)$ we have 
\begin{linenomath}
\begin{align*}
d_{eqGH}((X,d,G),(M,d_{g_N},G))&<\delta,\\
d_{eqGH}((M,d_{g_N},G),(M,d_{g_{N}(t)},G))&<\delta,\\
d_{eqGH}((M,d_{g_{N}(t)},G),(M_\infty,d_{g_\infty(t)},\Gamma_t))&<\delta.
\end{align*}
\end{linenomath}

Then using the triangle inequality we get that $d_{eqGH}((X,d,G),((M_\infty,d_{g_\infty(t)},\Gamma_t))<3\delta$. That is, $(M_\infty,d_{g_\infty(t)},\Gamma_t)$ converges to $(X,d,G)$ in the equivariant-Gromov-Hausdorff sense as $t\to 0$. 

Again by \cite[Proposition 4.1]{Harvey2016}, there exists injective Lie group morphisms $\bar{\theta}_{t}\colon \Gamma_t\to G$. Thus we have for all $t\in (0,t_0)$ that
\[
\dim(G)\geq \dim(\Gamma_t)\geq \dim G.
\]
Now by \cite[Theorem A]{AlAttar2024} we conclude that $(X,G)$ is equivariantly homeomorphic to $(M_\infty,\Gamma_t)$. That is, $\Gamma_t$ is isomorphic to $G$.
\end{proof}

\begin{proof}[Proof of \th\ref{MC: main corollary}]
 Observe that since $4\delta-1>0$ we can find $(4 \delta-1)/5>\varepsilon_0>0$ small enough fixed, such that $\delta-\varepsilon_0>1/4$. Then by \th\ref{T: Approximating bounded spaces} there exists a sequence of smooth Riemannian metrics $g_n^{\varepsilon_0}$ on $M=X$ such that $1/4<\delta-\varepsilon_0\leq \Sec(g_n^{\varepsilon_0})\leq 1+\varepsilon_0$. Observe now that by our choice of $\varepsilon_0$ we have
 \[
 	\frac{\delta-\varepsilon_0}{1+\varepsilon_0}>\frac{1}{4}.
 \]
Then assuming %{\color{red}
$c=1+\varepsilon_0$, we have that the metrics $cg_n^{\varepsilon_0}$ have
\[
\frac{1}{4}<\Sec(c g_n^{\varepsilon_0})\leq 1.
\]
Observe that $(M,cg_n^{\varepsilon_0})$ converges in the Gromov-Hausdorff sense to $(X,cd)$ and $\curv(cd)\geq \delta/(1+\varepsilon_0)$. Since $(4\delta-1)/5>\varepsilon_0$ we have that  $\delta/(1+\varepsilon_0)>1/4$. 

Thus by \cite[Theorem 1]{BrendleSchoen2009} it follows that $M$ is diffeomorphic to a round spherical space form $(\Sp^m(1)/\Gamma$ for some $\Gamma<\mathrm{O}(m+1)$, and thus the conclusion follows.
\end{proof}

\bibliographystyle{siam}
\bibliography{References}

\end{document}